%% file: main.tex
\documentclass[reqno]{amsart}

\usepackage[utf8]{inputenc}
\usepackage{amssymb}
\usepackage{amsthm}
\usepackage{amsmath}
\usepackage{tikz}

\usepackage[toc,page]{appendix}

\usepackage{enumitem}

\usepackage{tikz-cd}

\usepackage{tikz-cd}

\usepackage{tikz}
\usetikzlibrary{datavisualization}
\usetikzlibrary{datavisualization.formats.functions}

\usetikzlibrary{arrows}

\usepackage{import}

\usepackage{mathtools}
\usepackage{mathrsfs}
\usepackage{commath}

\usepackage{framed}

\usetikzlibrary{calc}

\makeatletter
\pgfarrowsdeclare{X}{X}
{
  \pgfutil@tempdima=0.3pt%
  \advance\pgfutil@tempdima by.25\pgflinewidth%
  \pgfutil@tempdimb=5.5\pgfutil@tempdima\advance\pgfutil@tempdimb by.5\pgflinewidth%
  \pgfarrowsleftextend{+-\pgfutil@tempdimb}
  \pgfarrowsrightextend{0pt}
}
{
  \pgfutil@tempdima=0.3pt%
  \advance\pgfutil@tempdima by.25\pgflinewidth%
  \pgfsetdash{}{+0pt}
  \pgfsetroundcap
  \pgfsetmiterjoin
  \pgfpathmoveto{\pgfqpoint{-5.5\pgfutil@tempdima}{-6\pgfutil@tempdima}}
  \pgfpathlineto{\pgfqpoint{5.5\pgfutil@tempdima}{6\pgfutil@tempdima}}
  \pgfpathmoveto{\pgfqpoint{-5.5\pgfutil@tempdima}{6\pgfutil@tempdima}}
  \pgfpathlineto{\pgfqpoint{5.5\pgfutil@tempdima}{-6\pgfutil@tempdima}}
  \pgfusepathqstroke 
}
\makeatother

\newtheorem{theorem}{Theorem}[section]
\newtheorem{lemma}[theorem]{Lemma}
\newtheorem{proposition}[theorem]{Proposition}
\newtheorem{corollary}[theorem]{Corollary}
\newtheorem{definition}[theorem]{Definition}

\theoremstyle{remark}
\newtheorem{remark}[theorem]{Remark}

\usepackage{relsize}
\usepackage{exscale}

\usepackage{mathtools}
\usepackage{mathrsfs}

\usepackage{nicefrac}

\usepackage{framed}

\usepackage{datetime}

\usepackage{hyperref} 

\renewcommand{\epsilon}{\varepsilon}

\DeclareMathOperator{\Orb}{Orb}

\title{On Inverse Shadowing} 
\author{Chris Good}
\author{Joel Mitchell}
\author{Joe Thomas}
\date{May 2019}

\begin{document}

\hypersetup{pageanchor=false} 

\subjclass[2000]{37B05, 37B10, 37B20, 54H20}
\keywords{inverse shadowing, minimality, sensitivity, equicontinuity}

\begin{abstract}
We give a reformulation of the inverse shadowing property with respect to the class of all pseudo-orbits. This reformulation bears witness to the fact that the property is far stronger than might initially seem. We give some implications of this reformulation, in particular showing that systems with inverse shadowing are not sensitive. Finally we show that, on compact spaces, inverse shadowing is equivalent to a finite version of it.
\end{abstract}

\maketitle

\hypersetup{pageanchor=true} 
\input{Sections/Introduction2.tex}
\input{Sections/Preliminaries.tex}

\input{Sections/Reformulation.tex}

\input{Sections/Implications.tex}

\bibliographystyle{plain} 
\bibliography{bib}

\end{document}

%% file: Sections/Introduction2.tex
Let $f\colon X\to X$ be a continuous function on a (not necessarily compact) metric space $X$. A $\delta$\textit{-pseudo-orbit} is a sequence $(x_i)_{i\in \omega}\subseteq X$ such that $d(f(x_{i}),x_{i+1}) <\delta$ for all $i\in\omega$. Such sequences arise naturally during the computation of orbits of points in dynamical systems, and so their study is of great importance for numerics. Indeed, through calculation one often encounters rounding errors and so the generated sequence of points is not actually a true orbit of the system, but instead a $\delta$-pseudo-orbit with $\delta$ dependent upon the degree of accuracy to which one can compute. One question that may then be asked is, to what extent does this sequence reflect any of the original dynamics in the system? Ultimately, this is a question regarding the stability of a dynamical system and one line of enquiry is to determine if such sequences are closely followed by true orbits of the system; thus leading directly to the notion of shadowing. The sequence $(y_i)_{i\in\omega}\subseteq X$ is said to $\varepsilon$-shadow the sequence $(x_i)_{i\in\omega}$ provided that $d(y_i,x_i)<\varepsilon$ for all $i\in\omega$. If $(y_i)_{i\in\omega}=(f^i(y))_{i\in\omega}$ for some $y\in X$, then this shadowing means that to some degree of accuracy $\varepsilon>0$, the pseudo-orbit is followed by a true orbit. If this happens for all pseudo-orbits of some given accuracy, then the system $(X,f)$ is said to have the shadowing property.  

This concept of shadowing has a natural interpretation when modelling a system numerically and this has been studied in detail in the works of Corless \cite{Corless}, Palmer \cite{Palmer} and Pearson \cite{Pearson}. It is also an important theoretical concept. For example, Bowen \cite{bowen-markov-partitions} used shadowing implicitly as a key step in his proof that the nonwandering set of an Axiom A diffeomorphism is a factor of a shift of finite type. Since then it has been studied extensively as a key factor in stability theory \cite{Pilyugin, robinson-stability,WaltersP}, in understanding the structure of $\omega$-limit sets and Julia sets 
\cite{
barwell-davies-good,
BarwellGoodOprochaRaines, 
BarwellMeddaughRaines2015, 
BarwellRaines2015,
Bowen, 
MeddaughRaines}, and as a property in and of itself \cite{Coven, fernandez-good, GoodMeddaugh2018, LeeSakai, Nusse, Pennings, Pilyugin,Sakai2003}.

In this paper, we look at the related concept, inverse shadowing. First introduced by Corless and Pilyugin \cite{CorlessPilyugin} as something akin to the ``dual'' of shadowing, and as part of the concept of bishadowing by Diamond \textit{et al} \cite{DiamondKozyakinKloedenPokrovskii}, inverse shadowing in a system informally means that true trajectories may be recovered from computed orbits (within some given accuracy).  
Kloeden, Ombach and Pokroskii \cite{KloedenOmbackPokrovskii} later defined inverse shadowing using the notion of a $\delta$-method which are functions mapping points to $\delta$-pseudo-orbits originating from the point. This allows one to consider certain classes of maps from the space to the space of pseudo-orbits through the imposition of extra structure on such mappings such as continuity.

 Such classes have been studied in a variety of different settings for example \cite{GoodMitchellThomas,Honary, KloedenOmbackPokrovskii,Lee, LeePark, Pilyugin2002}.  
Of particular interest has been its relationship to structural stability. In \cite{Pilyugin2002}, Pilyugin showed that if an Axiom A diffeomorphism on a closed $C^\infty$ manifold is structurally stable, then it has the inverse shadowing property with respect to classes of continuous methods. Meanwhile in \cite{KloedenOmbach}, Kloeden and Ombach prove that a structurally stable homeomorphism on a compact space has inverse shadowing with respect to the class of methods induced by homeomorphism. Further results in this direction can be found in \cite{ChoiLeeZhang,Honary,Lee}. 

Within this paper, we examine inverse shadowing with respect to the class of all pseudo-orbits. In Section \ref{SectionReform} we prove that inverse shadowing is equivalent to a property which essentially involves a quantifier swap in the definition. We use this reformulation to give a number of implications in Section \ref{SectionImplications}, in particular showing that a system exhibiting inverse shadowing is not \textit{eventually sensitive} (defined below). We conclude by showing that when the phase space is compact, for the classes $\mathcal{T}_0$, $\mathcal{T}_c$ and $\mathcal{T}_h$, inverse shadowing is equivalent to a finite version.

Throughout $(X,d)$ is a metric space: we emphasise that, unless otherwise stated, we do not assume $X$ to be compact. 
We denote by $\mathbb{Z}$ the set of all integers; the set of positive integers $1,2,3,4,\ldots$ is denoted by $\mathbb{N}$ whilst $\omega \coloneqq\mathbb{N}\cup\{0\}$.

%% file: Sections/Preliminaries.tex
\section{Preliminaries}

We firstly outline the notions of inverse shadowing that will be used here. Let $f \colon X \to X$ be a homeomorphism (resp. a continuous function). We call the pair $(X,f)$ a \textit{dynamical system}. The \textit{orbit} of $x$ under $f$ is the set of points $\{f^i(x)\}_{i\in A}$ and is denoted  by $\Orb_f(x)$, with the understanding that $A=\mathbb{Z}$ if $f$ is a homeomorphism and a full version of inverse shadowing is under consideration whilst $A=\omega$ when it is a positive version of inverse shadowing under consideration. For $\delta>0$, we refer to a bi-infinite sequence $(x_k)_{k\in \mathbb{Z}}$ such that $d(f(x_k), x_{k+1}) < \delta$ for all $k \in \mathbb{Z}$ as a \textit{$\delta$-pseudo-orbit}; and a mono-infinite sequence $(x_k)_{k\in \omega}$ such that $d(f(x_k), x_{k+1}) < \delta$ for all $k \in \omega$ as a \textit{positive $\delta$-pseudo-orbit}. 

Let $X^\mathbb{Z}$ (resp. $X^\omega$) be the product space of all bi-infinite (resp. mono-infinite) sequences, with the product topology (note that compactness of $X$ implies compactness of the product). Then for any given $\delta>0$, let $\Phi_f(\delta) \subseteq X^\mathbb{Z}$ be the set of all $\delta$-pseudo-orbits with respect to $f$ (resp. $\Phi ^+ _f(\delta) \subseteq X^\omega$ the set of all positive $\delta$-pseudo-orbits with respect to $f$). A mapping $\varphi \colon X \to \Phi_f(\delta)$ (resp. $\varphi\colon X\to\Phi_f^+(\delta)$), such that, for each $x \in X$, $\varphi(x)_0=x$, a (resp. positive) $\delta$-method for $f$, where $\varphi(x)_k$ is used to denote the $k$\textsuperscript{th} term in the sequence $\varphi(x)$. We denote by $\mathcal{T}_0(f, \delta)$ the set of all respective $\delta$-methods with understanding from the context whether this refers to positive methods or not. Similarly we denote by $\mathcal{T}_c(f,\delta)$ the set of all continuous (positive) $\delta$-methods , and by $\mathcal{T}_h(f,\delta)$ the set of all (positive) $\delta$-methods induced by a homeomorphism, i.e., $\mathcal{T}_h(f,\delta)$ is the set of (resp. positive) $\delta$-methods $\varphi$ for which there exists a homeomorphism $h \colon X \to X$ such that $d(f(x),h(x))< \delta$ for each $x \in X$ and $\varphi(x)_k=h^k(x)$, for all relevant $k \in \mathbb{Z}$ (resp. $k \in \omega$).

\begin{definition}
Let $(X,d)$ be a metric space and let $f \colon X \to X$ be a homeomorphism (resp. continuous function). We say that $f$ experiences \textit{(resp. positive) $\mathcal{T}_\alpha$-inverse shadowing}
if, for any $\epsilon>0$ there exists $\delta>0$ such that for any $x \in X$ and any $\varphi \in \mathcal{T}_\alpha(f,\delta)$ there exists $y \in X$ such that $\varphi(y)$ $\epsilon$-shadows $x$; i.e.
\[ \forall k \in \mathbb{Z} \,( \textit{resp. } k \in \omega), \,d(\varphi(y)_k, f^k(x))< \epsilon.\]
\end{definition}

\begin{definition}
Let $f \colon X \to X$ be a homeomorphism (resp. continuous function) on a metric space $X$. We say that $f$ experiences \textit{(resp. positive) weak inverse shadowing} with respect to the class $\mathcal{T}_\alpha$ if, for any $\epsilon>0$ there exists $\delta>0$ such that for any $x \in X$ and any $\varphi \in \mathcal{T}_\alpha(f,\delta)$ there exists $y \in X$ such that 
\[\varphi(y) \subseteq B_\epsilon\left(\Orb_f(x)\right).\]
(NB. As stated in the preliminaries, $\Orb_f(x)$ is only the positive trajectory of $x$ when considering the `positive' version of the above statement). 
\end{definition}

\begin{remark}
Clearly if a system has $\mathcal{T}_\alpha$-inverse shadowing then it has weak inverse shadowing with respect to the class $\mathcal{T}_\alpha$.
\end{remark}

%% file: Sections/Reformulation.tex
\section{An Equivalent Reformulation of Inverse Shadowing with respect to the class of all pseudo-orbits}\label{SectionReform}
Within the literature surrounding inverse shadowing, several authors \cite{CorlessPilyugin, Lee} have commented on the importance of restricting one's attention to certain admissible classes of pseudo-orbits. Theorem \ref{Reform} brings to light exactly why such a restriction may be important; in particular, it demonstrates that $\mathcal{T}_0$-inverse shadowing is a much stronger property than it may appear \textit{prima facie}.

\begin{theorem}\label{Reform}
A system $(X,f)$ with $f$ a homeomorphism (resp. continuous map) has (resp. positive) $\mathcal{T}_0$-inverse shadowing if and only if for any $\epsilon>0$ there exists $\delta>0$ such that for any $x \in X$ there exists $y \in X$ such that for any $\varphi \in \mathcal{T}_0(f,\delta)$ and any $k \in \mathbb{Z}$ (resp. $k \in \omega$) $d(f^k(x),\varphi(y)_k)<\epsilon$.
\end{theorem}

\begin{proof}
Clearly the latter implies the former. Thus, suppose that $f$ has (resp. positive) $\mathcal{T}_0$-inverse shadowing and assume further that the latter is false, that is, there exists an $\varepsilon>0$ such that
\begin{align}\label{eqnIS0.1}
\forall \delta>0 \, \exists x \in X : \forall y \in X \, \exists\varphi \in \mathcal{T}_0(f,\delta) \, \exists k \in \mathbb{Z} \, (\text{resp. } k \in \omega) : d(f^k(x), \varphi(y)_k)\geq \epsilon.
\end{align}

Take such an $\epsilon$ and let $\delta>0$ correspond to this $\epsilon$ in the definition of inverse shadowing. Now, fix $x \in X$ corresponding to this $\epsilon$ and $\delta$ as in (\ref{eqnIS0.1}). Then by (\ref{eqnIS0.1}), for each $y \in X$, there will exist a $\varphi_y \in \mathcal{T}_0(f,\delta)$ and a $k\in\mathbb{Z}$ (resp. $k\in\omega$) for which $d(f^k(x), \varphi_y(y)_k)\geq \epsilon$. Now, define a map $\varphi \colon X \to X^\mathbb{Z} \colon y \mapsto \varphi_y(y)$. By construction, $\varphi \in \mathcal{T}_0(f,\delta)$ and by (resp. positive) $\mathcal{T}_0$-inverse shadowing, there will exist $y \in X$ such that $d(f^k(x),\varphi(y)_k)<\epsilon$ for all $k \in \mathbb{Z}$ (resp. $k\in \omega)$. But by the construction of $\varphi$, this can not be the case for all such $k$ and thus one obtains a contradiction.
\end{proof}

When $X$ is compact we can use the uniform continuity of $f$ to obtain the following corollary.

\begin{corollary}\label{ReformCtd} Let $X$ be a compact metric space. A system $(X,f)$ with $f$ a homeomorphism (resp. continuous map) has (resp. positive) $\mathcal{T}_0$-inverse shadowing if and only if for any $\epsilon>0$ there exists $\delta>0$ such that for any $x \in X$ there exists $y \in X$ such that for any $\varphi \in \mathcal{T}_0(f,\delta)$ and any $k \in \mathbb{Z}$ (resp. $k \in \omega$)  and any $z \in B_\delta(y)$, $d(f^k(x),\varphi(z)_k)<\epsilon$.
\end{corollary}

\begin{proof}
That the latter entails the former is trivial. Therefore, suppose the former. Let $\epsilon>0$ be given. Then, by Theorem \ref{Reform} there exists $\eta>0$ (without loss of generality $\eta<\frac{\varepsilon}{2}$) such that for any $x \in X$, there exists $y \in X$ such that for any $\varphi \in \mathcal{T}_0(f,\eta)$ and any $k \in \mathbb{Z}$ (resp. $k \in \omega$) $d(f^k(x),\varphi(y)_k)<\frac{\epsilon}{2}$.

First consider the case when $f$ is a homeomorphism. Using uniform continuity, take $\delta>0$ such that, for any $a,b \in X$, if $d(a,b)<\delta$ then $d(f(a),f(b)) < \frac{\eta}{2}$ and $d(f^{-1}(a),f^{-1}(b))< \frac{\eta}{2}$; without loss of generality $\delta<\frac{\eta}{2}$. Now let $(z_k)_{k \in \mathbb{Z}}$ be a $\delta$-pseudo-orbit where $z_0 \in B_\delta(y)$. Then, by the triangle inequality, \[(\ldots, z_{-k},\ldots, z_{-1}, y, z_1, \ldots, z_k, \ldots)\] is a $\eta$-pseudo-orbit through $y$. It follows by the choice of $y$ that $d(f^k(x),z_k)<\frac{\epsilon}{2}$ for all $k \in \mathbb{Z}\setminus\{0\}$. Since $d(x,z_0) \leq d(x,y)+d(y, z_0)=\frac{\epsilon}{2}+\eta<\epsilon$, the sequence $(z_k)_{k \in \mathbb{Z}}$ $\epsilon$-shadows $x$ and we are done.

Now suppose that $f$ is a continuous map; we are now considering positive $\mathcal{T}_0$-inverse shadowing. We argue similarly to before. Using uniform continuity, take $\delta>0$ such that, for any $a,b \in X$, if $d(a,b)<\delta$ then $d(f(a),f(b)) < \frac{\eta}{2}$. Without loss of generality $\delta<\frac{\eta}{2}$. Now let $(z_k)_{k \in \omega}$ be a $\delta$-pseudo-orbit where $z_0 \in B_\delta(y)$. Then, by the triangle inequality, $(y, z_1, \ldots, z_k, \ldots)$ is a $\eta$-pseudo-orbit through $y$. It follows by the choice of $y$ that $d(f^k(x),z_k)<\frac{\epsilon}{2}$ for all $k \in \omega\setminus\{0\}$. Since \[d(x,z_0) \leq d(x,y)+d(y, z_0)<\frac{\epsilon}{2}+\delta<\frac{\epsilon}{2}+\eta<\epsilon,\] the sequence $(z_k)_{k \in \omega}$ $\epsilon$-shadows $x$ and we are done.
\end{proof}

Using the same technique as in the proof of Theorem \ref{Reform} we come by the following reformulation of weak inverse shadowing. We omit the proof.

\begin{theorem}\label{ReformWeak}
A map $f\colon X \to X$ has (resp. positive) weak inverse shadowing with respect to the class  $\mathcal{T}_0$ if and only if for any $\epsilon>0$ there exists $\delta>0$ such that for any $x \in X$ there exists $y \in X$ such that for any $\varphi \in \mathcal{T}_0(f,\delta)$, \[\varphi(y) \subseteq B_\epsilon\left(\Orb_f(x)\right).\]
\end{theorem}

%% file: Sections/Implications.tex
\section{Implications}\label{SectionImplications}

Before giving some implications of Theorem \ref{Reform}, we will need to recall some standard definitions in dynamical systems. 
A dynamical system $(X,f)$ is said to be \textit{equicontinuous} at a point $x \in X$ if for any $\epsilon>0$ there exists $\delta>0$ such that, for any $y \in X$, if $d(x,y)<\delta$ then, for any $n \in \omega$, $d(f^n(x),f^n(y))<\epsilon$. The system itself is said to be \textit{equicontinuous} if it is equicontinuous at every point. We observe that, when $X$ is compact, equicontinuity is equivalent to \textit{uniform equicontinuity} ($\delta$ is chosen independently of the point $x\in X$). A dynamical system exhibits \textit{sensitive dependence on initial conditions} (or is \textit{sensitive}) if there exists $\delta>0$ such that for any nonempty open set $U$ there exist $x,y \in U$ and $k \in \mathbb{N}$ such that $d(f^k(x),f^k(y))\geq \delta$. Such a $\delta$ is referred to as a \textit{sensitivity constant} for $(X,f)$. A weakening of sensitivity was introduced in \cite{GoodLeekMitchell}; a dynamical system $(X,f)$ is \textit{eventually sensitive} if there exists $\delta>0$ such that for any $x \in X$ and any $\epsilon>0$ there exist $n ,k \in \omega$ and $y \in B_\epsilon(f^n(x))$ such that $d(f^{n+k}(x), f^k(y)) \geq \delta$. We refer to such a $\delta$ as an \textit{eventual-sensitivity constant}. Clearly sensitivity implies eventual sensitivity but, as demonstrated in \cite{GoodLeekMitchell}, the converse is not true. It is also easy to see that neither sensitivity nor eventual sensitivity can be held in conjunction with equicontinuity. Finally a system is \textit{expansive} if there exists $\delta>0$ such that for any $x,y \in X$ there exists $k \in \omega$ with $d(f^k(x),f^k(y))\geq \delta$. It is an easy exercise to show that if a system is \textit{perfect}, that is the space has no isolated points, then an expansive system is sensitive. 

In \cite[Theorem 4]{Honary} the authors show a homeomorphism with $\mathcal{T}_0$-inverse shadowing is not expansive. Whilst not explicitly stated there, we remark that this result assumes that for any $\epsilon$ there exist distinct points $x,y \in X$ with $d(x,y) < \epsilon$. Indeed, it is easy to see that a system consisting of a single periodic orbit has $\mathcal{T}_0$-inverse shadowing but is also expansive. The reformulation given by Theorem \ref{Reform} enables us to give the following much stronger result.

\begin{theorem}\label{thmInvShadNotEvSens}
Let $X$ be a metric space and $f\colon X \to X$ a continuous map. If $f$ has the positive $\mathcal{T}_0$-inverse shadowing property then the system $(X,f)$ is not eventually sensitive. (If $f$ is a homeomorphism with $\mathcal{T}_0$-inverse shadowing, then neither $(X,f)$ nor the inverse system $(X,f^{-1})$ are eventually sensitive.)
\end{theorem}

\begin{proof}
Suppose that $f$ is eventually sensitive and let $\delta_0>0$ be a constant of eventual sensitivity. Pick $\epsilon>0$ with $\epsilon< \frac{\delta_0}{2}$. Now take a $\delta>0$ corresponding to this $\epsilon$ as in the reformulation of positive inverse shadowing in Theorem \ref{Reform}. Pick $x \in X$. Then there exists $y \in X$ such that every $\delta$-pseudo-orbit through $y$ $\epsilon$-shadows $x$. Consider $f(y)$: By eventual sensitivity there exist $n,k \in \mathbb{N}$ and $z \in B_\delta\left(f^{n+1}(y)\right)$ such that $d(f^k(z), f^{n+k+1}(y))\geq \delta_0$. But the sequence 
\[(y,f(y),\ldots, f^{n}(y), z, f(z), \ldots f^k(z), f^{k+1}(z), \ldots),\]
is a $\delta$-pseudo orbit from $y$. Similarly so is the orbit sequence of $y$. Therefore, by inverse shadowing, $f^{n+k+1}(y),f^k(z) \in B_\epsilon\left(f^{n+k+1}(x)\right)$. It follows by the triangle inequality that
\[d\left(f^{n+k+1}(y),f^k(z)\right) < 2\epsilon < \delta_0,\]
which is a contradiction. Therefore $f$ is not eventually sensitive.
\end{proof}

\begin{remark}
Note that there is no need to make the assumption that  for any $\epsilon>0$ there are points $\epsilon$-close in Theorem \ref{thmInvShadNotEvSens}.
\end{remark}

In \cite[Theorem 3]{Honary} the authors show that a chain transitive homeomorphism on a compact metric space is minimal if and only if it has the weak inverse shadowing property with respect to $\mathcal{T}_0$. With Theorem \ref{ReformWeak} this result, as well as the following analogous one, becomes much more elementary. (NB. In the statement of \cite[Theorem 3]{Honary} the authors assume the phase space is compact, this assumption does not appear necessary for their proof nor ours.) Recall first that a system $(X,f)$ is said to be minimal if, for $A \subseteq X$ closed, $f(A)=A$ implies that $A=X$ or $A=\emptyset$. Equivalently it is minimal if $\overline{\Orb_f(x)}=X$ for all $x \in X$.

\begin{theorem}[see \cite{Honary}]\label{thmInvShadTTEq}
Let $X$ be a metric space and $f \colon X \to X$ be a chain transitive continuous function. Then $f$ is minimal if and only if $f$ has positive weak inverse shadowing with respect to the class $\mathcal{T}_0$.
\end{theorem}

\begin{proof}
Since the orbit of every point is dense under a minimal map, it is trivial that such a map has positive weak inverse shadowing with respect to the class $\mathcal{T}_0$.

Pick $z \in X$. Let $\epsilon>0$ and take $\delta>0$ as in Theorem \ref{ReformWeak}. Then there exists $y \in X$ such that for any $\varphi \in \mathcal{T}_0(f,\delta)$, \[\varphi(y) \subseteq B_\epsilon\left(\Orb(z)\right).\]
Since $f$ is chain transitive there is a $\delta$-chain from $y$ to every point in $X$. It follows then by the above that $x \in B_\epsilon\left(\Orb(z)\right)$ for all $x \in X$. Since $\epsilon>0$ was arbitrary the result follows. 
\end{proof}

\begin{corollary}
Let $X$ be a metric space and let $f \colon X \to X$ be a chain transitive continuous function. If $f$ has the positive $\mathcal{T}_0$-inverse shadowing property then the system $(X,f)$ is equicontinuous. (If $f$ is a homeomorphism with $\mathcal{T}_0$-inverse shadowing, then both $(X,f)$ and the inverse system $(X,f^{-1})$ are equicontinuous.)
\end{corollary}
\begin{proof}
By Theorem \ref{thmInvShadTTEq} the system is minimal. A minimal system is either equicontinuous or sensitive (see\footnote{The authors of \cite{AuslanderYorke} prove that compact metric minimal systems are either uniformly equicontinuous or sensitive. Without the presence of compactness it is still an easy exercise to show that the system is either equicontinuous or sensitive.} \cite[Corollary 2] {AuslanderYorke}). 
\end{proof}

We conclude this paper by showing that, for compact systems, inverse shadowing with respect to any given class, is equivalent to what we call \textit{finite inverse shadowing}, with respect to the same class. This is akin to the result that shadowing is equivalent to finite shadowing in a compact system  \cite{Pilyugin}.

\begin{definition}
A system $(X,f)$ exhibits \textit{finite (resp. finite positive) $\mathcal{T}_\alpha$-inverse shadowing} if for any $\epsilon>0$ there exists $\delta>0$ such that for any $x \in X$ and any $\varphi \in \mathcal{T}_0(f,\delta)$ and any $n \in \mathbb{N}$ there exists $y_n \in X$ such that for any $k \in \{-n,\ldots, 0, \ldots, n\}$ (resp. $k \in \{0, \ldots, n\}$) $d(f^k(x),\varphi(y_n)_k)<\epsilon$.
\end{definition}

The proof of the lemma below is very similar to that of Theorem \ref{Reform} and is therefore omitted.

\begin{lemma}\label{ReformFinite}
A map $f\colon X \to X$ has finite (resp. finite positive) $\mathcal{T}_0$-inverse shadowing if and only if for any $\epsilon>0$ there exists $\delta>0$ such that for any $x \in X$ and any $n \in \mathbb{N}$ there exists $y_n \in X$ such that for any $\varphi \in \mathcal{T}_0(f,\delta)$ and any $k \in \{-n,\ldots, 0, \ldots, n\}$ (resp. $k \in \{0, \ldots, n\}$) $d(f^k(x),\varphi(y)_k)<\epsilon$.
\end{lemma}

\begin{proposition}\label{Finite} Let $X$ be a compact metric space. For any $\alpha \in \{0,c,h\}$, a homeomorphism (resp. continuous map) $f \colon X \to X$ has finite (positive) $\mathcal{T}_\alpha$-inverse shadowing if and only if it has (positive) $\mathcal{T}_\alpha$-inverse shadowing.
\end{proposition}

\begin{proof}
We will prove the cases when $f$ is a homeomorphism, with reference to $\mathcal{T}_\alpha$-inverse shadowing. The cases when $f$ is a continuous map, with reference to positive $\mathcal{T}_\alpha$-inverse shadowing, are similar.

Clearly, for any $\alpha \in \{0,c,h\}$, if $f$ has it has (positive) $\mathcal{T}_\alpha$-inverse shadowing then it has finite (positive) $\mathcal{T}_\alpha$-inverse shadowing.

Suppose that $f \colon X \to X$ has finite $\mathcal{T}_c$-inverse shadowing. Let $\epsilon>0$ be given and let $\delta>0$ correspond to $\frac{\epsilon}{2}$ in finite inverse shadowing. Let $x\in X$ and take $\varphi \in \mathcal{T}_c(f,\delta)$. For any $n \in \mathbb{N}$, there exists $y_n\in X$ such that $d(\varphi(y_n)_i,f^i(x))<\frac{\epsilon}{2}$ for all $i \in \{-n \ldots, 0, \ldots, n\}$. By sequential compactness, $(y_n)$ has a convergent subsequence; call it $(y_n)$ again. Let $y=\lim_{n\to \infty} y_n$; $y \in \overline{B_{\frac{\epsilon}{2}}(x)}$. Since $\varphi \colon X \to X^\mathbb{Z}$ is continuous, 
\[\lim_{n\to \infty}\varphi(y_n)=\varphi\left(\lim_{n\to \infty}y_n\right)=\varphi(y).\]
Moreover it follows that, for any $k \in \mathbb{N}$ there exists $M \in \mathbb{N}$ such that for any $m>M$ we have $d(\varphi(y)_i, \varphi(y_m)_i)<\frac{\epsilon}{2}$ for all $i \in \{-k, \ldots, 0, \ldots, k\}$. For a fixed $k$ and corresponding $M$, we must also have, for all $m>M$ that, $d(\varphi(y_m)_i,f^i(x))<\frac{\epsilon}{2}$ for all $i \in \{-k, \ldots, 0, \ldots, k\}$. Then, by the triangle inequality,

\[\forall k \in \mathbb{N} \exists M \in \mathbb{N} : \forall m>M \, \forall i \in \{-k, \ldots, 0, \ldots, k\}, \, d(\varphi(y)_i, f^i(x))<\epsilon.\]
From this it follows that for every $k \in \mathbb{N}$ we have $d(\varphi(y)_k, f^k(x))<\epsilon$.
Hence $f \colon X \to X$ has $\mathcal{T}_c$-inverse shadowing. The proof for the class $\mathcal{T}_h$ is similar to the above (as are the positive versions for the classes $\mathcal{T}_c$ and $\mathcal{T}_h$); these are therefore omitted.

Now suppose that $f \colon X \to X$ has finite $\mathcal{T}_0$-inverse shadowing. Let $\epsilon>0$ be given and let $2\delta>0$ correspond to $\frac{\epsilon}{2}$ in finite inverse shadowing; without loss of generality taking $\delta < \frac{\epsilon}{4}$. Then, by Lemma \ref{Reform}, for any $n \in \mathbb{N}$ there exists $y_n \in X$ such that for any $\varphi \in \mathcal{T}_0(f,2\delta)$, and any $i \in \{-n, \ldots,0,\ldots, n\}$, $d(\varphi(y_n)_i, f^i(x))< \frac{\epsilon}{2}$. By sequential compactness, $(y_n)$ has a convergent subsequence; call it $(y_n)$ again. Let $y=\lim_{n\to \infty} y_n$; $y \in \overline{B_{\frac{\epsilon}{2}}(x)}$. Because $f$ and $f^{-1}$ are continuous, there exists $N \in \mathbb{N}$ such that $d(f(y_n),f(y))<\delta$ and $d(f^{-1}(y_n),f^{-1}(y))<\delta$ for all $n>N$. Then for any $n>N$ we have

\[B_\delta\left(f(y)\right)\subseteq B_{2\delta}\left(f(y_n)\right)\subseteq B_{\frac{\epsilon}{2}}\left(f(x)\right), \]
and
\[B_\delta\left(f^{-1}(y)\right)\subseteq B_{2\delta}\left(f^{-1}(y_n)\right)\subseteq B_{\frac{\epsilon}{2}}\left(f^{-1}(x)\right). \]

Now let $(z_k)_{k\in \mathbb{Z}}$ be a $\delta$-pseudo-orbit through $y$ (so $z_0=y$), we will show this $\epsilon$-shadows $x$. Fix $n>N$. Then \[(\ldots,...,z_{-k},\ldots,z_{-2}, z_{-1},y_n,z_1,z_2,\ldots,z_{k},\ldots),\] is a $2\delta$-pseudo-orbit through $y_n$; by finite shadowing $d(f^i(x), z_i) \leq \frac{\epsilon}{2}$ for all $i \in \{-n,\ldots, -1,1,\ldots, n\}$. Additionally, $d(x,z_0)\leq \frac{\epsilon}{2}<\epsilon$, as $z_0=y$. Since $n>N$ was arbitrary it follows that, for any $n>N$, $d(f^i(x), z_i) \leq \frac{\epsilon}{2}$ for all $i \in \{-n,\ldots, -1,1,\ldots, n\}$. This in turn entails that $d(f^i(x), z_i) \leq \frac{\epsilon}{2}$ for all $i \in \mathbb{Z}$. As this was an arbitrary $\delta$-pseudo-orbit through $y$ we are done.
\end{proof}